\documentclass[12pt]{article}

\usepackage[usenames]{color} 

\usepackage{amssymb} 
\usepackage{amsmath} 
\usepackage{amssymb,bm,amsmath,amssymb,amsthm}
\usepackage[]{inputenc}
\usepackage[small,nohug,heads=vee]{diagrams}
\diagramstyle[labelstyle=\scriptstyle]

\usepackage{authblk}

\usepackage[all]{xy}
  \SelectTips{cm}{10}     
  \everyxy={<2.5em,0em>:} 

\usepackage{hyperref}

\DeclareMathOperator{\Gal}{Gal}

\DeclareMathOperator{\GCD}{gcd}

\DeclareMathOperator{\ord}{ord}

\DeclareMathOperator{\CL}{Cl}
\DeclareMathOperator{\tor}{tors}
\DeclareMathOperator{\Hom}{Hom}

\DeclareMathOperator{\T}{\mathbb T}

\DeclareMathOperator{\C}{\mathcal C}
\DeclareMathOperator{\I}{\mathcal I}
\DeclareMathOperator{\PP}{\mathcal P}
\DeclareMathOperator{\Pl}{Pl}

\newtheorem{cor}{Corollary}
\newtheorem{ex}{Example}

\newtheorem{lem}{Lemma}
\newtheorem{defin}{Definition}
\newtheorem{theorem}{Theorem}
\newcommand{\overbar}[1]{\mkern 1.5mu\overline{\mkern-1.5mu#1\mkern-1.5mu}\mkern 1.5mu}

\begin{document}
\title{On Abelianized Absolute Galois Group of Global Function Fields}
\author[*]{Bart de Smit \\ 
\texttt{desmit@math.leidenuniv.nl}
}

\author[*]{\\Pavel Solomatin \\
\texttt{p.solomatin@math.leidenuniv.nl}
}

\affil[*]{Leiden University, Mathematical Department\\
Niels Bohrweg 1, 2333 CA Leiden}

\date{ Leiden, 2017}
\maketitle

\begin{abstract}
The main purpose of this paper is to describe the abelian part $\mathcal G^{ab}_{K}$ of the absolute Galois group of a global function field $K$ as pro-finite group. We will show that the characteristic $p$ of $K$ and the non $p$-part of the class group of $K$ are determined by $\mathcal G^{ab}_{K}$. The converse is almost true: isomorphism type of $\mathcal G_K^{ab}$ as pro-finite group is determined by the invariant $d_K$ of the constant field $\mathbb F_q$ introduced in first section and the non $p$-part of the class group.
\end{abstract}

\textbf{\\ \\ \\ \\ \\ Acknowledgements:} 
This paper is a part of the PhD research of the second author under scientific direction of the first author. The second author was supported by the ALGANT scholarship during this research. Both authors would like to thank professors Hendrick Lenstra and Peter Stevenhagen for helpful discussions during the project.

\newpage
\section{Introduction}
Let $K$ be a global function field, i.e. field of functions on a smooth projective geometrically connected curve $X$ defined over a finite field $\mathbb F_q$, where $q=p^n$, $p$ is prime. The famous theorem of Uchida \cite{Uchida} states that the geometric isomorphism class of $X$ is determined by the isomorphism class of the absolute Galois group $\mathcal G_K = \Gal(K^{sep}:K)$ considered as topological group. One of the essential step in the Uchida's proof is to recover from $\mathcal G_K$ its abelian part $\mathcal G_K^{ab}$ with some additional data, like decomposition and inertia subgroups. The following questions are natural to ask: what kind of information one could recover from the isomorphism class of the pro-finite abelian group $\mathcal G^{ab}_{K}$? More concretely, does the abelian part of the absolute Galois group determine the global function field $K$ up to isomorphism? If not which function fields share the same $\mathcal G^{ab}_{K}$ for some fixed isomorphism class of $\mathcal G^{ab}_{K}$?

For a global function field $K$ of characteristic $p$ with the exact constant field $\mathbb F_{q}$, $q=p^n$ we define the invariant $d_K$ as a natural number such that  $n = p^{k} d_{K}$ with $\gcd(d_{K}, p)=1$. Let $\CL^{0}(K)$ denotes the degree zero part of the class-group of $K$. In other words $\CL^{0}(K)$ is the abelian group of $\mathbb F_q$-rational points of the Jacobian variety associated to the curve $X$. 
The main purpose of this paper is to prove the following result: 

\begin{theorem}\label{main}
Suppose $K$ and $K'$ are two global function fields, then $\mathcal G^{ab}_K \simeq \mathcal G^{ab}_{K'}$ as pro-finite groups if and only if the following three conditions hold:
\begin{enumerate}
	\item $K$ and $K'$ share the same characteristic $p$;
	\item Invariants $d_K$ and $d_{K'}$ coincide: $d_K = d_{K'}$; 
	\item The non $p$-parts of class-groups of $K$ and $K'$ are isomorphic: $$\CL^{0}_{non-p} (K) \simeq \CL^{0}_{non-p} (K').$$
\end{enumerate}
In particular, two function fields with the same exact constant filed $\mathbb F_q$ have isomorphic $\mathcal G^{ab}_{K}$ if and only if they have isomorphic $\CL^{0}_{non-p} (K)$.
\end{theorem}

This theorem provides some answers to the above questions. 
 For example, we have: 
\begin{cor}\label{corl1}
Let $K$ be the rational function field (with genus zero) over fixed constant field $\mathbb F_q$ and let $E$ be an elliptic function field (with genus one) defined over the same constant field, such that\footnote{the existence of such field is guaranteed by the Waterhouse theorem, see the last section.} $\# \CL^{0}(E) =  q$. Then there exists isomorphism of topological groups $\mathcal G^{ab}_{K} \simeq \mathcal G^{ab}_{E}$. In particular, the genus $g$ of $K$ and therefore the Dedekind zeta-function $\zeta_K(s)$ of $K$  are not determined by $\mathcal G^{ab}_{K}$ even if the constant field $\mathbb F_q$ is fixed. 

\end{cor}

The above example also shows that:
\begin{cor}\label{corl2}
There are infinitely many function fields of the same characteristic $p$(but with different cardinality of the constant field) with isomorphic $\mathcal G^{ab}_{K}$.
\end{cor}
\begin{proof}
Fix a prime number $p$ and let $q= p^{p^{k}}$, where $k$ is a non-negative integer. Let $F_k$ and $E_k$ denote rational and elliptic function fields from the previous example with the exact constant field $\mathbb F_q$. Then, according to the our main theorem for any non-negative integers $k$, $l$ we have: $\mathcal G^{ab}_{K_{l}}~\simeq~\mathcal G^{ab}_{E_{k}}.$
\end{proof}

Applying some classical results about the two-part of $\CL^{0}(K)$ of hyper-elliptic function fields we will also show that:
\begin{cor}\label{corl3}
For any given $q$ with $p>2$ there are infinitely many isomorphism types of $\mathcal G_K^{ab}$ which could occur for function fields with the exact constant field $\mathbb F_q$. 
\end{cor}
\begin{proof}
See theorem \ref{maincor} from the last section.
\end{proof}

Unfortunately, the answer to the question about distribution of global fields over fixed constant field $\mathbb F_q$ sharing the same $\mathcal G^{ab}_{K}$ is not clear at the moment, since we don't know if there are infinitely many such fields with a given non $p$-part of the class group. In particular, it seems to be reasonable to state the following \textbf{conjecture:} \emph{there are infinitely many curves defined over fixed finite field $\mathbb F_q$, $q=p^n$ with the order of the group of $\mathbb F_q$-rational points of the Jacobian varieties associated to them to be a power of $p$}. If the conjecture is true then what is the proportion of such curves, say as $q$ fixed and $g$ tends to infinity?
 
Note that the proof of the theorem \ref{main} includes the explicit reconstruction of the invariants $p$, $d_K$ and $\CL^{0}_{non-p} (K)$ from $\mathcal G^{ab}_K$. More concretely, let $s_l(\mathcal G^{ab}_K)$ be the least integer $k$ such that $\mathcal G^{ab}_K$ has direct summand of the form $\mathbb Z/ l^{k}\mathbb Z$ and let $p^{\star}$ denotes $(-1)^{\frac{p-1}{2}} p $ if $p$ is odd and $p$ otherwise, then:
\begin{theorem}
Given the isomorphism class of the topological group $\mathcal G^{ab}_K$ we have: 
\begin{enumerate}
	\item The characteristic $p$ of $K$ is a unique prime such $\mathcal G^{ab}_K$ has no elements of order $p$; 
	\item The non-$p$  part $\CL^{0}_{non-p} (K)$ of the class-groups of $K$ is isomorphic to the torsion of the quotient $\mathcal G^{ab}_{K} / \overbar{\mathcal G^{ab}_{K}[\tor] } $, where $\overbar{\mathcal G^{ab}_{K}[\tor] }$ denotes the closure of the torsion subgroup of $\mathcal G^{ab}_K$ : $$\CL^{0}_{non-p} (K) \simeq (\mathcal G^{ab}_{K} / \overbar{\mathcal G^{ab}_{K}[\tor] })[\tor] .$$
	\item  The natural number $d_K$ is a unique number such that for any prime number $l \ne p$: $$\ord_l(d_K) = \begin{cases} 0, & \mbox{if } l=2 \mbox{ and } s_2(\mathcal G^{ab}_K)=1; \\ s_l(\mathcal G^{ab}_K) - \ord_l( (p^{\star})^{l-1} -1 ), & \mbox{ otherwise.} \end{cases}$$
	\end{enumerate}
\end{theorem}
\begin{proof}
See corollaries \ref{rec1} and $\ref{rec2}$.
\end{proof}

The main idea towards our result was inspired by the work \cite{Peter1}, where authors produced an elegant description for isomorphism class of the topological group $\mathcal G^{ab}_K $, where $K$ denotes \emph{imaginary quadratic number field}. But also, note that there are many completely different technical details, which give in some sense opposite to their result.

The paper has the following structure: in the next section we will sketch the proof of the theorem \ref{main}. Then we prove all the necessarily lemmas. Finally, we will discuss the question about construction of non-isomorphic function fields with isomorphic and non-isomorphic abelian parts of their absolute Galois groups. 

\section{Outline of the Proof }

The global class field theory provides an inherits description of the abelian part of the absolute Galois group of a global or local field $K$ in terms of different arithmetic objects associated to $K$. We will use the \emph{Idele}-theoretical approach, see section \ref{ClassField} for details and the following classical books \cite{Neuk}, \cite{WeilBook}, \cite{ArtinTate} for complete discussion. For a given global function field $K$ with the exact constant field $\mathbb F_q$, $q=p^n$, $p$ is prime let $\I_K$ denotes the group of Ideles of $K$ and $\C_K$ denotes the \emph{Idele class-group of $K$}, i.e. the quotient group of $\I_K$ by the multiplicative group $K^{\times}$. Then the Artin map provides us with the homomorphism $\C_K \to \mathcal G^{ab}_K $. In the function field case this map is injective, but not surjective, since $\C_K$ is not compact. But if we take the pro-finite completion(with respect to the given topology) of $\C_K$ then the main theorem of the class field theory says that we have isomorphism of topological groups: 
$  \widehat{\C_K} \simeq  \mathcal G^{ab}_K$, see for example [theorem 6, chapter 9 of \cite{WeilBook}]. 

Recall that we have a split exact sequence: $$0 \to \C^{0}_K \to \C_K  \to \mathbb Z \to 0,$$ where $\C^{0}_K$ is the degree zero part of the Idele class group and the map from $\C_K$ to $\mathbb Z$ is the degree map. Now, $\C^{0}_K$ is pro-finite, hence complete and therefore $\widehat{ \C_K} \simeq \C^{0}_K \oplus \widehat{\mathbb Z}$.  We will show that $\mathcal G^{ab}_K \simeq \mathcal G^{ab}_{K'}$ if and only if $\C^{0}_K \simeq \C^{0}_{K'}$. The key ingredient in the our proof is the Pontryagin duality for locally compact abelian groups, which allows us to reduce question about pro-finite abelian groups to the question about discrete torsion groups. 

\begin{lem}\label{zlem}
Let $A$ and $B$ be two pro-finite abelian groups. Then $A \simeq B$ if and only if $A \oplus \widehat{\mathbb Z} \simeq B \oplus \widehat{\mathbb Z}$ in the category of pro-finite abelian groups.
\end{lem} 
\begin{proof}See section \ref{proof1}.
\end{proof}
This lemma reduces our question to the description of $\C^{0}_K$ as topological group. Let $v$ denotes a place of $K$ and $K_v$, $\mathcal O_v$  denotes the corresponding completion and its ring of integers respectively. Then we derive the following exact sequence. 

\begin{lem}
There exists an exact sequence of topological groups, where all finite groups treated with the discrete topology: $$ 1\to \mathbb F^{\times}_q  \to \prod_{v} \mathcal O^{\times}_v \to \C^{0}_K \to \CL^{0}(K) \to 1.$$  
\end{lem}
\begin{proof}
See section \ref{proof2}.
\end{proof}

After that we recall the isomorphism $\mathcal O^{\times}_v \simeq \mathbb F^{\times}_{q^{n}} \times \mathbb Z_p^{\infty}$, where $n$ is the degree of a place $v$ and $\mathbb Z_p$ denotes the group of $p$-adic integers. Denoting by $\mathcal T_K$ the group $(\prod_{v} \mathbb F^{\times}_{q^{\deg(v)}} ) / \mathbb F^{\times}_q$ we will get the following exact sequence:

\begin{equation} \label{es}
1\to \mathcal T_K \times \mathbb Z_p^{\infty} \to \C^{0}_K \to \CL^{0}(K) \to 1  
\end{equation}

There are two crucial observation about this sequence. First we will prove the following structure theorem for the group $\mathcal T_K$:
\begin{theorem}\label{tortheorem}
Given a function field $K$ with the exact constant field $\mathbb F_q$, where $q=p^n$ there exists an isomorphism $\mathcal T_K \simeq \prod_{l,m} (\mathbb Z / l^{m} \mathbb Z)^{a_{l,m}}$, where the product is taken over all prime numbers $l$ and all positive integers $m$ and $a_{l,m}$ denotes a finite or countable cardinal number. Moreover, coefficients $a_{l,m}$ depend only on $q$ and the following holds:
\begin{enumerate}
	\item	Each $a_{l,m}$ is either zero or infinite countable cardinal;
	\item For $l=p$ we have $a_{p,m} = 0$ for all $m$;
	\item For $l \ne p$, $l \ne 2$ there exists a unique non-negative integer $N_q(l)$ such that $a_{l, m}$ is infinite if and only if $m \ge N_q(l)$;
	\item For $p \ne 2$ and $l=2$ there exists a unique non-negative integer $N_q(2)$ such that for $q = 1 \mod 4$ we have $a_{2, m}$ is infinite if and only if $m \ge N_q(2)$, and for $q = 3 \mod 4$ we have $a_{2, m}$ is infinite if and only if $m =1$ or $m \ge N_q(2)$; 
	\item Given two prime powers $q_1$, $q_2$ numbers $N_{q_1} (l)$ and $N_{q_2} (l)$ coincide for all $l$ if and only if $q_1=p^{n_1}$, $q_2=p^{n_2}$ with $\frac{n_1}{n_2} = p^{m}$, for some integer $m$. 
\end{enumerate}
\end{theorem}
\begin{proof}
See section \ref{proof3}.
\end{proof}

From now we denote the group $\mathcal T_K$ by $\mathcal T_q$.

\begin{defin}
The exact sequence of abelian groups $0 \to A \to B \to^{\psi} C \to 0$ is called \emph{totally non-split} if there is no non-trivial subgroup $S$ of $C$ such that the sequence $0 \to A \to \psi^{-1}(S) \to S \to 0$ splits. 	
\end{defin}  

The second observation is the key point in the our proof.

\begin{theorem}
All torsion elements of $\C^{0}_K$ are in $\mathcal T_q$. Therefore the exact sequence \ref{es} is totally non-split. Moreover, the topological closure of the torsion subgroup of $\C^{0}_K$ is   $\mathcal T_q$ : $\overline{\C^{0}_K [\tor]} = \mathcal T_q$. 
\end{theorem}
\begin{proof}
See section \ref{torsion}.
\end{proof}

Because of the description of $\mathcal T_q$ this theorem gives us:

\begin{cor}\label{rec1}
If $\mathcal G^{ab}_K \simeq \mathcal G^{ab}_{K'}$ as pro-finite groups then $K$ and $K'$ share the same group $\mathcal T_q$, in particular the characteristic $p$ and the invariant $d_K$ are determined by the isomorphism class of $\mathcal G^{ab}_K$. 
\end{cor}
\begin{proof}
Since $\mathcal G^{ab}_K \simeq \C^{0}_K \oplus \widehat{ \mathbb Z} $ and the group $\widehat{ \mathbb Z}$ is torsion free, we have that $\mathcal T_q$ is also the closure of the torsion subgroup of $\mathcal G^{ab}_K$. Then theorem \ref{tortheorem} shows that $p$ is a unique prime such that this group has no elements of order $p$.

For the natural number $d_K$ consider the torsion group $\mathcal G^{ab}_K [\tor]$. By the theorem $\ref{tortheorem}$ this group has direct summand of the form $\mathbb Z / l^k \mathbb Z$ for a fixed prime $l \ne p$ if and only if $k \ge N_q(l)$ or $l=2$, $k=1$, $p=3 \mod 4$ and $d_K = 1 \mod 2$. In the proof of the theorem $\ref{tortheorem}$ we will show that $N_q(l) = \ord_l(d_K) + \ord_l ( (p^{\star})^{l-1} -1 )$, where $p^{\star} = -p$ if $p = 3\mod 4$ and $p^{\star}=p$ otherwise. Which implies the formula: $$\ord_l(d_K) = \begin{cases} 0, & \mbox{if } l=2 \mbox{ and } s_2=1 \\ s_l(\mathcal G^{ab}_K) - \ord_l( (p^{\star})^{l-1} -1 ), & \mbox{ otherwise.} \end{cases}$$

\end{proof}

Since each pro-finite abelian group is isomorphic to the limit of finite abelian groups, by the Chinese remainder theorem it is also isomorphic to the product over prime numbers of its primary components. We will work with these components separately instead of working with the whole group. Let $l$ be a prime number different from $p$, we have: $1 \to \mathcal T_{q,l} \to \C^{0}_{K,l} \to \CL^{0}_{l}(K) \to 1$. Which shows that: $$\CL^{0}_{l}(K) \simeq \C^{0}_{K, l} / \overbar{ \C^{0}_{K, l}[\tor]}.$$

\begin{cor}\label{rec2}
If $\mathcal G^{ab}_K \simeq \mathcal G^{ab}_{K'}$ as pro-finite groups then the non $p$-parts of the class-groups of $K$ and $K'$ are isomorphic: $\CL^{0}_{non-p} (K) \simeq \CL^{0}_{non-p} (K').$
\end{cor}
\begin{proof}
We know that $\mathcal G^{ab}_K \simeq \C_{K}^{0} \oplus \widehat{\mathbb Z}$ and that $\mathcal T_q = \overbar{\mathcal G^{ab}_K[\tor] }$. Considering the $l$-part we get: 

$$\mathcal G^{ab}_{K,l}/ \overbar{\mathcal G^{ab}_{K,l}[\tor] } \simeq (\C_{K,l}^{0}/\mathcal T_{q,l} ) \oplus  {\mathbb Z_l}.$$ 
Since $\mathbb Z_l$ is torsion free, we have:
$$(\mathcal G^{ab}_{K,l}/ \overbar{\mathcal G^{ab}_{K,l}[\tor] })[\tor] \simeq \C_{K,l}^{0}/\mathcal T_{q,l} \simeq \CL^{0}_{l}(K).$$

Finally, note that the $p$-part of the torsion group of $\mathcal G^{ab}_K$ is trivial and hence combining all primes $l$ different from $p$ we get: 
$$\CL^{0}_{non-p} (K) \simeq (\mathcal G^{ab}_{K} / \overbar{\mathcal G^{ab}_{K}[\tor] })[\tor].$$

\end{proof}

These two results imply the only if part of the theorem \ref{main}. Now, we are going to discuss the question about the other implication. Our goal is to show that for a given $\mathcal T_q$ and given non $p$-part of the class group there is only one possibility for $\C^{0}_K$ to fit in the exact sequence \ref{es}.

Consider the $p$-part of the exact sequence \ref{es}: 
$$1\to \mathbb Z_p^{\infty} \to \C^{0}_{K, p} \to \CL^{0}_{p}(K) \to 1.  $$

By using the fact that this sequence is totally non-split we will show(see lemma \ref{ZpLem}) that this implies $\C^{0}_{K, p} \simeq \mathbb Z_p^{\infty}$, in particular $\mathcal G^{ab}_{K,p}$\emph{ doesn't depend on }$\CL^{0}_{p}(K)$.  

We fix a prime number $l \ne p$ and consider the $l$-part which is of course also totally non-split: 

\begin{equation} \label{es2}
1\to \mathcal T_{q,l} \to \C^{0}_{K, l}  \to \CL^{0}_{l} (K) \to 1 
\end{equation}

Obviously, if $\CL^{0}_{l} (K) \simeq 0$ then $\C^{0}_{K, l} \simeq \mathcal T_{q,l}$. Our goal is to show that even if $\CL^{0}_{l} (K)$ is not the trivial group then this totally non-split sequence determines $\C^{0}_{K, l}$ uniquely.  In order to achieve our goal we need the following:
\begin{theorem}\label{group}
Let $\{ C_i \}$ be a countable set of finite cyclic abelian $l$-groups with orders of $C_i$ are not bounded as $i$ tends to infinity and let $A$ be any finite abelian $l$-group.
Then up to isomorphism there exists and unique torsion abelian $l$-group $B$ satisfying two following conditions:
\begin{enumerate}
	\item There exists an exact sequence: $1 \to A \to B \to \oplus_{i \ge 1} C_i \to 1$;
	\item $A$ is the union of all divisible elements of $B$: $A = \cap_{n \ge 1} nB$.
\end{enumerate} 
\end{theorem}

Applying the Pontryagin duality to the exact sequence \ref{es2}  we get:

$$ 1\leftarrow (\mathcal T_{q,l})^{\vee} \leftarrow (\C^{0}_{K, l})^{\vee} \leftarrow (\CL^{0}_{l} (K))^{\vee} \leftarrow 1.
$$ 
We will show in corollary \ref{cor1} that this sequence dual to the sequence \ref{es2} satisfies conditions of the theorem \ref{group} and therefore $\C^{0}_{K, l}$ is uniquely determined, since its dual $(\C^{0}_{K, l})^{\vee}$ is uniquely determined.

\section{Proof of Lemmas}
In this section we are going to prove all the necessarily lemmas needed for our proof. Let us start from recalling some basic facts about pro-finite abelian groups. A standard references are \cite{Kapl} and \cite{Fuchs}.
\subsection{Preliminaries}
Let $A$ an abelian, not necessarily topological group. If this group is finitely generated then the structure theorem says that $A$ is isomorphic to $\mathbb Z^{r} \oplus A_{\tor}$ where $r$ is a non-negative integer called rank and $A_{\tor}$ is a finite abelian group. Given two such groups we have that they are isomorphic if and only if they ranks and torsion parts coincide. The structure of an infinitely generated abelian group is more complicated. An element $x$ of the abelian group $A$ is \emph{divisible} if for any $n \in \mathbb N$ there exists $y$ such that $x = ny$. A group $A$ is \emph{divisible} if all its elements are divisible. For example $\mathbb Q$ is divisible. Another example is the so-called \emph{Pr\"{u}fer $p$-group} which is defined as union of all $p^{k}$ roots of unity in $\mathbb C^{\times}$ for a fixed prime number $p$: $Z(p^{\infty}) = \{ \zeta \in \mathbb C^{\times} | \zeta^{p^{k}} = 1, k \in \mathbb N  \}$. Note that we have isomorphism of abstract groups: $Z(p^{\infty}) \simeq \mathbb Q_p / \mathbb Z_p$, where $\mathbb Q_p$ denotes the abelian group of $p$-adic numbers and $\mathbb Z_p$ is a subgroup of all $p$-adic integers.

A group is called \emph{reduced} if it has no divisible elements. 
\begin{lem}
Each abelian group $A$ contains the maximal divisible subgroup $D$ and is isomorphic to the direct sum of $D$ and some reduced group $R$: $A \simeq D \oplus R$.
\end{lem}
\begin{proof}
Proof of this and the following lemma could be find at the chapter 3 of the book \cite{Fuchs}. 
\end{proof}

The structure of the divisible subgroup is clear. 
\begin{lem}\label{Reduced}
Any divisible group $D$ is isomorphic to the direct sum of copies of $\mathbb Q$ and $Z(p^{\infty}) $.
\end{lem}

The structure of the reduced part is more complicated and usually involves the theory of Ulms invariants. In this paper we will work with the reduced part directly not referring to the Ulms invariants at all. 
\subsubsection{The Pontryagin Duality}
We need to recall some properties of the Pontryagin duality for locally compact abelian groups. A good reference including some historical discussion is \cite{Pont}. Let $\T$ be the topological group $\mathbb R / \mathbb Z$ given with the quotient topology. If $A$ is any locally compact abelian group one consider the Pontryagin dual $A^{\vee}$ of $A$ which is the group of all continuous homomorphism from $A$ to $\T$ : $$ A^{\vee} = \Hom(A, \T). $$ 
This group has the so-called compact-open topology and is a topological group. Here we list some properties of the Pontryagin duality we use during the proof: 
 \begin{enumerate}
 	\item The Pontryagin duality is a contra-variant functor from the category of locally compact abelian groups to itself;
	\item If $A$ is a finite abelian group treated with the discrete topology then $A^{\vee} \simeq A$ non-canonically;
	\item We have the canonical isomorphism: $(A^{\vee})^{\vee} \simeq A$;
	\item The Pontryagin dual to the pro-finite abelian group $A$ is a discrete discrete torsion group and vice versa;
	\item The Pontryagin duality sends direct products to direct sums and vice versa;
	\item The Pontryagin dual of $\mathbb Z_p$ is $Z(p^{\infty})$ and dual of $\mathbb Q / \mathbb Z$ is the group of pro-finite integers $\widehat{\mathbb Z}$.
 \end{enumerate}
 Having stated this we are able to prove our lemmas. 
 \begin{proof}[\bf{Proof of Lemma \ref{zlem}}]\label{proof1}
 Let $A$ and $B$ be two pro-finite abelian groups such that $A \oplus \widehat{\mathbb Z} \simeq B \oplus \widehat{\mathbb Z}$. Applying the Pontryagin duality to the above isomorphism we obtain: $$ (A)^{\vee} \oplus \mathbb Q / \mathbb Z \simeq  (B)^{\vee} \oplus \mathbb Q / \mathbb Z. $$  
Now since each abelian group is isomorphic to the direct sum of its reduced and divisible components and using the fact that $\mathbb Q/\mathbb Z$ is divisible we have that reduced part of $(A)^{\vee}$ and $(B)^{\vee}$ are isomorphic. Now, according to the Lemma \ref{Reduced} each divisible part  of $(A)^{\vee} \oplus \mathbb Q / \mathbb Z$ is direct sum of copies of $\mathbb Q$ and $Z(p^{\infty}) $ and since $\mathbb Q / \mathbb Z \simeq \oplus_{p} Z(p^{\infty})$ divisible parts of $(A)^{\vee}$ and $(B)^{\vee}$ are isomorphic. Therefore $(A)^{\vee}$ and $(B)^{\vee}$ are isomorphic and hence $A \simeq B$.    

 \end{proof}
 \subsection{Class Field Theory}\label{ClassField}
 We will start from the description of local aspects of the class field theory. Let $L$ be a local field of positive characteristic $p>0$. In other words $L$ is a completion of a global function field $K$ with respect to the discrete valuation associated to the place $v$ of $K$. This field is isomorphic to the field of Laurant series with constant field $\mathbb F_{q^{n}}$ and the corresponding ring of integers $\mathcal O_L$ is the ring of formal power series:  $L \simeq \mathbb F_{q^{n}} ((x))$, $\mathcal O_L \simeq \mathbb F_{q^{n}} [[x]]$. One way to construct abelian extensions of $L$ is to take the algebraic closure $\overbar{\mathbb F_{q^{n}}}$ of the constant field $\mathbb F_{q^{n}}$ which has Galois group $\Gal({\overbar{\mathbb F_{q^{n}}} : \mathbb F_{q^n}}) \simeq \widehat{\mathbb Z}$. This is the maximal unramified abelian extension of $L$. Another way to construct abelian extensions is to add to $L$ a root of the equation $x^{e} - t=0$, where $e$ is natural number such that $\gcd(e, p) =1$. This is totally ramified extension. Denoting by $I_{L} = \Gal^{ram}( {L^{ab} : L} )$ the inertia subgroup of $\mathcal G^{ab}_{L}$ we have the following split exact sequence: $$1 \to I_{L}  \to \mathcal G^{ab}_{L} \to \widehat{\mathbb Z} \to 1.$$ Recall that we also have the split exact sequence given via the valuation map: $$1 \to \mathcal O^{\times}_L \to L^{\times} \to \mathbb Z \to 1 .$$ The crucial point is that the local Artin map: $L^{\times} \to \mathcal G^{ab}_{L} $ induces   isomorphism of topological groups between the completion $\widehat{L^{\times}}$ of $L^{\times}$ and $\mathcal G^{ab}_{L}$ such that two exact sequences are isomorphic:  
 
 \[\xymatrix@=1.1em{
  1\ar[r] &  \widehat{\mathcal O^{\times}_L}\simeq \mathcal O^{\times}_L \ar[r]\ar[d] & \widehat{L^{\times}} \ar[r]\ar[d] & \widehat{\mathbb Z} \ar[d]\ar[r] & 1 \\
       1\ar[r] & I_{L} \ar[r] & \mathcal G^{ab}_{L} \ar[r] & \Gal({\overbar{\mathbb F_{q^n}} : \mathbb F_{q^n}})  \ar[r] & 1 \\
}\]

Now if $K$ is a global function field then it is possible to give a similar description of $\mathcal G^{ab}_{K}$ via Idele-class group. Let $\I_K$ denotes the multiplicative group of Ideles of $K$. This is the restricted direct product $\I_K = \prod'_{v} K^{\times}_{v}$, this product is taken over places $v$ of $K$ with respect to $\mathcal O_v^{\times}$. One defines the basic open sets as $ U = \prod'_{v} U_v $, where $U_v$ open in $K^{\times}_v$ and almost all $U_v = \mathcal O_v^{\times}$. Under the topology generated by such $U$ this becomes a topological group. The multiplicative group $K^{\times}$ is embedded to $\I_K$ diagonally as discrete subgroup and the quotient $\C_K$ is the \emph{ Idele class group} of $K$. This is a topological group, but not pro-finite. One defines the global Artin map $\C_K \to \mathcal G^{ab}_{K}$ . This map is injective, but not surjective. Similar to the local case it induces isomorphism of the pro-finite completion of $\C_K$ and $\mathcal G^{ab}_{K}$ as topological groups: $\widehat{\C_K} \simeq \mathcal G^{ab}_{K}$. 
 
 \subsection{Deriving the main exact sequence}
 
 Now our goal is to derive the exact sequence \ref{es}. Let $\I^{0}_K$ be the group of degree zero Ideles of $K$. It means the kernel of the degree map from $\C_K$ to $\mathbb Z$. We have : 
 $$ 1\to K^{\times} \to \I^{0}_K \to \C^{0}_K \to 1. $$
 
Let $\PP(K)$ denotes the group of ideals of $K$ and $\PP^{0}(K)$ be the kernel of the degree map to $\mathbb Z$. We also have:
$$ 1\to \mathbb F_q^{\times} \to  K^{\times}  \to \PP^{0}(K) \to \CL^{0}(K) \to 1. $$
 
There is a surjective homomorphism $\alpha$ of topological groups from $\I^{0}_K $ to $\PP^{0}(K)$, sending an Idele $(a_{P_1}, a_{P_2} , \dots)$ to the divisor $\sum v_{P_i}(a_{P_i}) \cdot P_i$. This is well-defined since for a given Idele almost all $a_{P} \in \mathcal O^{\times}_{v_{P}}$. The kernel of this map is $\prod_v \mathcal O^{\times}_v$. Moreover, this map sends principal Ideles to principal ideals and hence induces the surjective quotient map $\hat{\alpha}$ from $ \C^{0}_K$ to $\CL^{0}(K)$. We have the following snake-lemma diagram:  
\[\xymatrix{
  & 1 \ar[d]  &1 \ar[d]  & 1 \ar[d]\\ 
 1\ar[r] & \mathbb F_q^{\times} \ar[r]\ar[d] & \prod_v \mathcal O^{\times}_v \ar[r]\ar[d] & \ker \hat{\alpha} \ar[d]\\
 1\ar[r] & K^{\times} \ar[r]\ar[d]\hole
      & \I^{0}_K \ar[r]\ar[d]\hole
      & \C^{0}_K \ar[r]\ar[d]_(.6)(.3)\hole="hole" & 1 \\
  1\ar[r] & K^{\times}/ \mathbb F_q^{\times} \ar[r]\ar[d] & \PP^{0}(K) \ar[r]\ar[d] & \CL^{0}(K) \ar[d] \ar[r] & 1\\
  & 1 \ar[r]
      \ar@{<-} `l [lu]-(.5,0) `"hole" `[rrruu]+(.5,0) `[rruuu] [rruuu]
      & 1 \ar[r] & 1 \\
}\]

And therefore we have: 

$$1\to \mathbb F_{q}^{\times} \to \prod_v \mathcal O^{\times}_v  \to \C^{0}_K \to \CL^{0}(K) \to 1. $$ \label{proof2}

\subsection{On the Structure of the Kernel}
Now we will give an explicit description of the group $\ker{\hat{\alpha}} \simeq (\prod_v \mathcal O^{\times}_v )/ \mathbb F_q^{\times} $. If $v$ is a place of degree $n$ a global function field $K$ with the exact constant field $\mathbb F_q$, then $K_v$ is the field of Laurant series with constant field $\mathbb F_{q^{n}}$ and $\mathcal O_v$ is the ring of formal power series:  $K_v \simeq \mathbb F_{q^{n}} ((x))$, $\mathcal O_v \simeq \mathbb F_{q^{n}} [[x]]$. A formal power series is invertible if and only if it has non-zero constant term and therefore: $$\mathcal O^{\times}_v \simeq \mathbb F^{\times}_{q^{n}} \times (1 + t \mathbb F_{q^{n}}[[t]]).$$

\begin{lem}
We have isomorphism of topological groups: $1 + t \mathbb F_{q^{n}}[[t]] \simeq \prod_{\mathbb N} \mathbb Z_p$.
\end{lem}
\begin{proof}
See \cite{Neuk}, section on local fields.
\end{proof}

Denoting by $\mathcal T_q$ the group $(\prod_v \mathbb F^{\times}_{q^{\deg(v)}}) / \mathbb F^{\times}_q$, we obtain: 

$$ (\prod_v \mathcal O^{\times}_v )/ \mathbb F_q^{\times} \simeq \mathcal T_q \times \mathbb Z_p^{\infty}  .$$

\subsubsection{Description of $\mathcal T_q$}\label{proof3}

At the first time it seems that the group $\mathcal T_q$ depends on $K$ since the product $\prod _v \mathcal O^{\times}_v$ is taken over all places of $K$. Our first goal is to show that it actually depends only on $q$. \begin{bf} Recall that: \end{bf} because of the  Weil-bound each function field $K$ has places of all except finitely many degrees.

Consider the group $A_q = \prod_v \mathbb F^{\times}_{q^{\deg(v)}}$.  By the Chinese reminder theorem we have: $$A_q= \prod (\mathbb Z/ l^{m} \mathbb Z)^{a_{l,m}}, $$ 
where $a_{l,m}$ is either a non-negative integer or infinity. Since $\mathbb F^{\times}_{q^{n}}$ is a cyclic group of order $q^{n}-1$ we have the direct description of $a_{l,m}$: it is the cardinality of the set $  \{ v \in \Pl(K)  |  \deg(v)=n, \ord_l (q^{n} - 1) = m   \} $. Note that $a_{p, m} = 0$ for all $m \in \mathbb N$. 

\begin{lem}
Each $a_{l,m}$ is either $0$ or infinity. 
\end{lem}

\begin{proof}
Indeed, suppose that there exists a place $v$ of degree $n$ such that $ \ord_l (q^{n} - 1) = m$, we would like to show that then there are infinitely many such $v$. Our assumption is equivalent to the statement that $q^{n} = 1 \mod l^{m}$, but $q^{n} \ne 1 \mod l^{m+1}$. The order of the group $(\mathbb Z / l^{m+1} \mathbb Z)^{\times}$ is $\phi(l^{m+1}) = l^{m+1} - l^{m}$, where $\phi(a)$ denotes the Euler $\phi$-function. It means if $q^{n}$ satisfies our condition then for any $k \in \mathbb N$ the quantity $q^{n+ k\phi(l^{m+1})}$ also satisfies our condition. In other words, this condition depends only on $n \mod \phi(l^{m+1})$. Since each function field $K$ has places of all except finitely many degrees if there is one $v$ then there are infinitely many. 
\end{proof}

Now, given $l \ne p$ we would like to understand how many $m$ such that $a_{l, m} = 0$ do we have. First we will prove the following elementary number theory lemma.
\begin{lem}\label{element}
Let $a$ be a positive integer such that $\ord_l(a -1) =n$ for some prime number $l$. Then if $l\ne 2$ or $n \ge 2$ we have $\ord_l(a^{l} -1) = n+1$. 
\end{lem}
\begin{proof}
By the assumption of the lemma there exists an integer $b$ such that $\gcd(b,l)=1$ and  $a = 1 + bl^{n} \mod l^{n+1}$. Suppose that $l \ne 2$. For some integer $c$ we have: $$a^{l} = (1 + bl^{n} + cl^{n+1})^{l} = 1+ l(bl^{n} + cl^{n+1}) + \frac{l(l-1)}{2}(bl^{n} + cl^{n+1})^2 + \dots  = $$ $$=1+ l^{n+1}(b+cl) + \frac{l(l-1)}{2} l^{2n} (b+cl)^2 + \dots.  $$  Since $l \ne 2$ we have $a^{l} = 1+ bl^{n+1} \mod l^{m+2}$. 

Now let $l=2$ and $n\ge 2$. We have: $a = 1 + 2^{n} + b2^{n+1 } \mod 2^{n+2}$ and therefore $a^2 = 1 + 2^{n + 1} \mod 2^{n+2}$.
\end{proof}

\begin{lem}
For each odd prime number $l$ different from $p$ there exists $N(l)$ such that $a_{l,m}$ is infinite if and only if $m \ge N(l)$. Moreover $N(l)$ depends only on $q$ and doesn't depend on $K$. 
\end{lem}
\begin{proof}
 Let $d=l-1$ and $N(l) = \ord_l(q^{d} -1)$. Then $q^{d} = 1$ in the group $(\mathbb Z/ l^{N(l)} \mathbb Z)^{\times}$, but $q^{d} \ne 1$
in the group $(\mathbb Z/ l^{N(l)+1} \mathbb Z)^{\times}$. Therefore, for each $u \in \mathbb N$ such that $u = d \mod \phi(l^{N(l)+1})$ we have: $\ord_l (q^{u} - 1)=N(l)$. Since $K$ has places of almost all degrees the set $  \{ v \in \Pl(K)  |  \deg(v)=d \mod \phi(l^{N(l)+1}) \} $ is  infinite and hence $a_{l, N(l)} \ne 0$. 
We would like to show that if $a_{l,m} \ne 0$ then $a_{l, m+1} \ne 0$. We know that there exists a place of the degree $d_0$ such that $\ord_l(q^{d_0} -1) =m$. By the previous lemma  we have $\ord_l(q^{ld_0} -1) = m+1$. Then for any place $v$ from the set $\{ v \in \Pl(K)  |  \deg(v)=ld_0 \mod \phi(l^{m+2}) \}$ we have $\ord_l (q^{\deg(v)} -1) = m+1$. This shows that if $m \ge N(l)$ then $a_{l, m}$ is infinite. 

The last step is to show that $a_{l, m} =0$ if $m$ is less than $\ord_l(q^{d} -1)$. Indeed, the order $a$ of $q$ in the group $\mathbb F^{\times}_l$ divides $(l-1)$ and then $\ord_l (q^{a} - 1 ) = \ord_l (q^{a\frac{l-1}{a}} -1) = \ord_l (q^{l-1} -1)$, since $\frac{l-1}{a}$ is co-prime to $l$. It means that if for some $u$ we have $q^{u} = 1 \mod l$, then $u=a b$ and $\ord_l(q^{u} - 1) =\ord_l(q^{ab} - 1)  \ge \ord_l(q^a -1) = \ord_l (q^{l-1} -1)$.

\end{proof}

\begin{lem}
For $l=2$ the following holds. 
\begin{enumerate}
	\item If $p=2$, then $a_{m,2} = 0$ for all $m$; 
	\item if $q=1 \mod 4$, then there exists $N(2)$ such that $a_{2,m}$ is infinite if and only if $m \ge N(2)$; 
	\item if $q=3 \mod 4$, then there exists $N(2)$ such that $a_{2,m}$ is infinite if and only if $m \ge N(2)$ or $m=1$;
\end{enumerate}
\end{lem}
\begin{proof}
The first statement is trivial. For the second one let $N(2) = \ord_2(q-1)$, then $N(2) \ge 2$. As before we have $q = 1 \mod 2^{N(2)}$, but $q\ne 1 \mod 2^{N(2)+1 }$. The group $(\mathbb Z/ 2^{N(2)+1} \mathbb Z)^{\times}$ has order $\phi(2^{N(2)+1})$ and hence, for each $m$ such that $m=1 \mod \phi(2^{N(2)+1})$ we have that $q^{m} = 1 \mod 2^{N(2)}$, but $q\ne 1 \mod 2^{N(2)+1 }$. Since $K$ has places of almost all degrees the set $  \{ v \in \Pl(K)  |  \deg(v)=1 \mod \phi(l^{N(2)+1}) \} $ is  infinite and hence $a_{2, N(2)} \ne 0$. Now, as in the previous lemma if $a_{l, m} \ne 0$, then $a_{l, m+1}$ is not zero\footnote{ here we use the fact that $m \ge 2$.} and obviously if $m < N(2)$ we have $a_{2,m} =0$.

Finally suppose that $q= 3 \mod 4$. By the same argument as before we have that $a_{2,1}$ is infinite, but then $q^2 = 1 \mod 8$ and hence $a_{2,2} = 0$. Let $N(2)= \ord_2(q^2 -1) \ge 3$. We have that for $a_{2, N(2)}$ is infinite and for all $k$ such that $1 < k < N(2)$ we have $a_{2, k} = 0$. Because of the same argument as before $a_{2, m}$ is infinite for all $m \ge N(2)$. 

\end{proof}

In order to show that $T_q \simeq A_q$ we need one elementary lemma.

\begin{lem}
For a given prime power $q$ there are infinitely many integer numbers $n$ such that $\gcd(\frac{q^{n}-1}{q-1} , q-1) =1$.
\end{lem}
\begin{proof}
Consider the factorization of $q-1$ into different prime factors: $q-1 = l_1^{k_1} \dots l_m^{k_m}$. We know that $q = 1 \mod l_i^{k_i}$ and $q \ne 1 \mod l_i^{k_i+1}$, for all $i$ in $\{1,\dots, m \}$. In other words there exists a natural number $a_i$ co-prime to $l_i$ such that $q = 1+ a_i l_i^{k_i} \mod l_i^{k_i+1}$. Therefore if the natural number $n$ is co-prime to $q-1$ then $q^{n} = 1 + a_in l_i^{k_i} \mod l_i^{k_i + 1}$ and then $\gcd(\frac{q^{n}-1}{q-1} , q-1)=1$.
\end{proof}

\begin{cor}
We have isomorphism $A_q \simeq \mathcal T_q$. The characteristic $p$ of the constant field of $K$ is determined by $\mathcal T_q$. 
\end{cor}
\begin{proof}
For the first statement recall that $\mathbb F^{\times}_q$ is embedded diagonally to the product $\prod_v \mathbb F^{\times}_{q^{\deg(v)}}$. Now pick any prime $\beta$ of $K$ of degree $m$  such that $\gcd(\frac{q^{m}-1}{q-1} , q-1) =1$ and split the last product into two parts $ \mathbb F^{\times}_{q^{m}} \oplus \prod_{v \ne \beta} \mathbb F^{\times}_{q^{\deg(v)}} $. Note that $\mathbb F^{\times}_q$ is a subgroup of $\mathbb F^{\times}_{q^{m}}$ which is direct summand. Since all these groups have the discrete topology, the quotient $ \prod_{v \ne \beta} \mathbb F^{\times}_{q^{\deg(v)}} \oplus (\mathbb F^{\times}_{q^{n}} / \mathbb F^{\times}_q )$ is topologically isomorphic to $\mathcal T_q$. Finally, since each $a_{n,l}$ is either zero or infinity we have that $A_q \simeq \mathcal T_q$.

For the second statement note that $p$ is unique prime such that $a_{p, m} = 0$ for all $m \in \mathbb N$.
  
\end{proof}

\begin{lem}
For odd prime number $l$ we have $N(l) = \ord_l(p^{l-1} -1) + \ord_l{d_K}$.

\end{lem}
\begin{proof}
Recall the isomorphism  $\mathbb Z_l^{\times} \simeq (\mathbb Z_l)^{\times}_{\tor} \times(1+l\mathbb Z_l)$, for $l$ be an odd prime number. The multiplicative group $1+l\mathbb Z_l$ has the following filtration: $$1 \subset 1+l\mathbb Z_l \subset 1+l^2 \mathbb Z_l \subset \dots$$
Given $q$ and $l \ne p$ let $d$ be the order of $q \in \mathbb  (\mathbb Z_l)^{\times}_{\tor}$. Then by our construction $N(l)$ is the greatest integer such that $q^{d} \in 1+l^{N(l)} \mathbb Z_l$. Raising $q$ to the power $p$ doesn't change the filtration. On the other hand, lemma \ref{element} shows that raising $q$ to the power $l$ shifts the filtration exactly at one. Hence for $q=p^{d_K p^{n}}$, $\gcd(d_K, p) = 1$ we have:

$$ N(l) = \ord_l (q^{l-1} -1) = \ord_l (p^{(l-1)d_K p^{k}} -1)= \ord_l (p^{l-1} -1 ) +\ord_l(d_K) $$ 

\end{proof}

Recall that for a prime number $l$ different from $p$ we define $s_l(\mathcal T_q)$ to be the least integer $k$ such that $T_q$ has direct summand of the form $\mathbb Z/ l^{k}\mathbb Z $. Obviously, if $l \ne 2$ then $s_l(T_q) = N(l)$. More generally, we have: 

\begin{lem}\label{dkforml}
 The natural number $d_K$ is a unique number such that for any prime number $l \ne p$: $$\ord_l(d_K) = \begin{cases} 0, & \mbox{if } l=2 \mbox{ and } s_2=1 \\ s_l(\mathcal T_q) - \ord_l( (p^{\star})^{l-1} -1 ), & \mbox{ otherwise.} \end{cases}$$
\end{lem}
\begin{proof}
The case of the odd $l$ is clear, since $p^{\star} = (-1)^{\frac{p-1}{2}}p$ if $p$ is odd and hence for $l=1\mod 2$ we have $(p^{\star})^{l-1} = p^{l-1}$. If $l=2$ then there are two cases. If $p=1 \mod 4$ then $s_2(T_q) = N(2)$ and obviously $p^{\star} = p$, hence our formula holds trivially. If $p=3 \mod 4$ then either $q=3 \mod 4$ or $q =1 \mod 4$. In the first case we have $d_K = 1 \mod 2$ and $s_2(\mathcal T_q) =1$ which leads to the our "exceptional case". In the second case we have $d_K = 0 \mod 2$ and then $N(2) = s_2(T_q) \ge 2$ and hence $s_2(T_q) = \ord_2(q -1) = \ord_2 (p^{p^{k}d_K} -1 ) = \ord_2( p^{2 \frac{d_K}{2}} -1) = \ord_2 (p^{2} -1 ) + \ord_2(d_K) -1 = \ord_2(p+1) +\ord_2(d_K) = \ord_2(p^{\star} -1) +\ord_2(d_K)$, since in this case $p^{\star} = - p$.   
\end{proof}

Now we are able to prove our main result concerning isomorphism type of the abelian group $\mathcal T_q$. For a prime power $q = p^{n}$ we define $d_q$ to be the non-$p$ part of $n$ : $d_q = \frac{n}{p^{\ord_p{n}}}$. Trivially, for a function field $K$ with the exact constant field $\mathbb F_q$ we have $d_K = d_q$.

\begin{theorem}
Given two powers of $p$: $q_1= p^{n_1}$ and $q_2=p^{n_2}$ groups $\mathcal T_{q_1}$ and $\mathcal T_{q_2}$ are isomorphic if and only if $d_{q_1} = d_{q_2} $.
\end{theorem}
\begin{proof}
The only invariants of $T_q$ are the sequence of coefficients $a_{l, m}$ for different $l$, $m$. We will show that they coincide for all $l$, $m$ if and only if the condition of the our theorem holds. 

First we will prove the if part. We know that $d_{q_1} = d_{q_2}$. Let $l$ be an odd prime number different from $p$, then by the formula from the above lemma $s_l(\mathcal T_{q_1}) = s_l(\mathcal T_{q_2})$ and we have $a_{l, m} =0$ if and only if $m < s_l(\mathcal T_{q_1})$ which shows that coefficients $a_{m,l}$ coincide for $\mathcal T_{q_1}$ and $\mathcal T_{q_2}$. Suppose that $l=2$. If $p=2$ then $a_{2, m} =0$ for all $m$ in both groups. If $p = 1 \mod 4$ or $d_{q_1} = 0 \mod 2 $ then as before $a_{2,m} = 0$ if and only if $m < N_2(l) = s_2(T_{q_1}) = \ord_2(d_{q_1}) + \ord_2(p-1)$ and hence $a_{2,m}$ coincide for both groups. Finally, if $p=3\mod 4$ and $d_{q_1} = d_{q_2} = 1 \mod 2$ then $a_{2,m} =0$ if and only if either $m=1$ or  $m>N(2)=\ord_2(q_1^2-1) = \ord_2(q_2^2-1) $. The equality $\ord_2(q_1^2-1) = \ord_2(q_2^2-1)$ holds since: $\ord_2(q_1^{2} -1) = \ord_2(q_1 + 1) + 1 = \ord_2(p^{d_{q_1}p^{k}} +1 ) + 1 = \ord_2(p+1) + 1$.

Now, suppose that $T_{q_1} \simeq T_{q_2}$. Then by the formula from lemma \ref{dkforml} for any odd prime number $l$ different from $p$ we have $\ord_l(d_{q_1}) = \ord_l(d_{q_2})$. By definition we have $\ord_p(d_{q_1}) = \ord_p(d_{q_2}) = 0$. Finally, for $l=2$ there are two cases. Either both groups contain direct summand of the form $\mathbb Z/ 2\mathbb Z$ and then $\ord_2(d_{q_1}) =\ord_2(d_{q_2})=0$, or otherwise the formula from lemma \ref{dkforml} holds and then $\ord_2(d_{q_1}) = \ord_2(d_{q_2})$. 
\end{proof}

This already gives some important corollary. Let $q = 2^{2^{k}}$ for some non-negative integer $k$, then coefficients $a_{l, m}$ defined as follows: $$a_{l, m} = \begin{cases} \mathbb N, & \mbox{if } l \ne 2 \mbox{ and } m\ge \ord_l( 2^{l-1} -1 ) \\ 0  , & \mbox{ otherwise.} \end{cases}$$

\begin{cor}\label{trivial}
The following function fields of characteristic two share the same abelianized absolute Galois group $\mathcal G^{ab}_{K} \simeq \prod_{l, m} (\mathbb Z/ l^{m} \mathbb Z)^{a_{l,m}} \times \prod_{\mathbb N} \mathbb Z_2 \oplus \widehat{\mathbb Z}$ :
\begin{enumerate}
  	\item The rational function field  with $g=0$ over $\mathbb F_{2^{2^k}}$, for any non-zero integer $k$;
	\item The elliptic function field $y^2+y= x^3+x+1$, with $g=1$ over $\mathbb F_2$;
	\item The hyper elliptic function field $y^2+y= x^5+x^3 +1$, with $g=2$ over $\mathbb F_2$;
	\item The hyper elliptic function filed  $y^2 + y = (x^3 + x^2 + 1)(x^3 + x + 1)^{-1}$ ,  with $g=2$ over $\mathbb F_2$; 
	\item The function field of the plane quartic $y^4 + (x^3+x+1)y+(x^4+x+1)=0$, with $g=3$ over $\mathbb F_2$.
	\item The elliptic function field  $y^2+y=x^3 + \mu$, with $g=1$ over $\mathbb F_4$, where $\mu$ is the generator of $\mathbb F^{\times}_4$.
	
 \end{enumerate} 
In particular, the genus, the constant field and the zeta-function of $K$ are not determined by $G^{ab}_{K}$.  	   
\end{cor}
\begin{proof}
All these fields have trivial $\CL^{0}(K)$, see \cite{Func1}. Since $\mathcal T_2 \simeq \mathcal T_{2^{2^k}}$ we have that for any $K$ listed above $\C^{0}_K \simeq \mathcal T_2 \times \prod_{\mathbb N} \mathbb Z_2$.  
\end{proof}

\begin{bf}Remark: \end{bf} For given $q$ we will call a prime $l$ \emph{exceptional} if $N(l) > 1$. The question which $l$ are exceptional seems to be very difficult. Of course, if $l^2 | (q-1) $, then $a_{l, 1} =0$. For example if $q=9$ then $a_{2,1} = a_{2, 2} = 0$. But also there are exceptional primes $l$ with $\gcd(l , q-1)=1$. For example if $p=q=7$ and $l=5$. Then if $7^{d} = 1 \mod 5$ if and only if $d=4k$, but then $7^{d} =  49^{2k} = (-1)^{2k} = 1 \mod 25$. It means that $5$ is exceptional. We expect that for a given $q$ there are infinitely many exceptional primes, but we have no idea how to prove it even for the case $q=2$: the first exceptional prime for this case is $1093$. This phenomena is strictly related with the so-called \emph{Wieferich primes}.

\subsection{On the torsion of $\C^{0}_K$}\label{torsion}
Now our goal is to understand what happens with the exact sequence $1\to \mathcal T_q \times \mathbb Z_p^{\infty} \to \C^{0}_K \to \CL^{0}(K) \to 1$, when $ \CL^{0}(K)$ is not trivial. Since we are working with infinite groups $\C^{0}_K$ could still be isomorphic to the group mentioned in the corollary \ref{trivial}.  

\begin{theorem}
All the torsion of $\C^{0}_K$ are in $\mathcal T_q$, hence the exact sequence \ref{es} is totally non-split. Moreover, the topological closure of the torsion subgroup of $\C^{0}_K$ is   $\mathcal T_q$ : $\overline{\C^{0}_K [\tor]} = \mathcal T_q$. 
\end{theorem}
 
\begin{proof}
Suppose that there exists a non-zero $x \in \C^{0}_K$ such that $x^{l} = 0$ for some prime number $l$, not necessarily co-prime to $p$. We would like to show that actually this element has trivial image in the class group. Pick a representative $(x_1 , x_2 , \dots )$ for $x$ as element of $\mathcal I_K$ , we know that almost all $x_i \in \mathcal O^{\times}_v$ and that $x^l = (x_1^{l}, x_2^{l} , \dots)$ is a principal Idele. Let $a$ be the element of $K^{\times}$ whose image in $\mathcal I_K$ is $x^{l}$. We have that $a$ is locally an $l$-th power and hence by [theorem 1, chapter 9] from \cite{ArtinTate} we have that $a$ is globally an $l$-th power and hence $x$ is a principal Idele up to multiplication by the element $(\zeta_1, \zeta_2, \dots ) \in \mathcal T_q$, where $\zeta_i$ denotes an $l$-th root of unity and hence its image in the class group is trivial. 

Since $\mathbb Z_p$ is torsion free we have that all the torsion of $\C^{0}_K$ lies in $\mathcal T_q$. Note that each element of finite order in $\mathcal T_q$ is an element of the direct sum $\oplus_{l, m} (\mathbb Z/ l^{m} \mathbb Z)^{a_{l,m}} $ and closure of this direct sum is $T_q$ itself. 
\end{proof}

As it was mentioned in the introduction this statement implies the "only if" part of our main result. 

\subsection{Proof of the inverse implication}

Our task in this section is for given $K$ show that the data $\CL^{0}_{non-p} (K) , \mathcal T_{ q}$ determines $\C^{0}_K$ up to isomorphism. 
\subsubsection{The $p$-part}
Our first goal is to show that the $p$-part of $\C^{0}_K$ is isomorphic to $\mathbb Z_p^{\infty}$.

We start from an easy example. Consider the exact sequence: $0 \to \mathbb Z_p \to  \mathbb Z_p \to \mathbb Z/ p^{k}\mathbb Z \to 0 $, where the second map is multiplication by $p^{k}$. This sequence is totally non-split. We claim that $\mathbb Z_p$ is a unique group which could occur in the middle of this sequence. More concretely:

\begin{ex}
Let $A$ be an abelian pro-$p$ group such that the following sequence is totally non-split: $0 \to \mathbb Z_p \to  A \to \mathbb Z/ p^{k}\mathbb Z \to 0 $, then $A \simeq \mathbb Z_p$.
\end{ex}
\begin{proof}
Since $\mathbb Z_p$ is torsion free and the sequence is totally non-split then $A$ is also torsion free. Let us denote the quotient map by $\phi$. There exists $x \in A$ such that $\phi(x)$ is the generator of $\mathbb Z/ p^{k}\mathbb Z$. Moreover,  since $A$ is torsion free we know that $p^{k}x $ is a non-zero element $a=\phi(x)$ of $\mathbb Z_p$. We claim that the first non-zero coefficient in the $p$-adic expression $a=a_0 +a_1p +a_2p^2 + \dots$ is $a_0$. Indeed, if $a$ is divisible by $p$ then $p(x - a/p) = 0$ and hence $x=a/2$ since $A$ is torsion free. Then $A$ is generated by $\{x, \mathbb Z_p \}$ with the relation $p^{k}x =a$. Consider the map $\psi : A \to \mathbb Z_p$, which sends element $x$ to $a$ and $\mathbb Z_p \to p^{k}\mathbb Z_p$. Then $\psi$ is homomorphism: $ \psi (p^{k}x) = \psi(a) = p^ka = p^k \psi(x)$. The kernel of this map is trivial and since $a_0 \ne 0$ then this map is onto.  
\end{proof}

This example gives an idea how to prove the following:

\begin{lem}\label{ZpLem}
Let $A$ be an abelian pro-$p$ group such that the following sequence is totally non-split: $0 \to \mathbb Z_p^{\infty} \to  A \to B \to 0 $, where $B$ is a finite abelian $p$-group. Then $A \simeq \mathbb Z_p^{\infty}$.
\end{lem}
\begin{proof}
Since the sequence is totally non-split and $\mathbb Z_p$ is torsion free, then $A$ is torsion free also. This means that multiplication by any natural number is injective. It means that the Pontryagin dual $A^{\vee}$ of $A$ is torsion(since $A$ is pro-finite) and divisible(since the dual to the injection is surjection). Consider the dual sequence: $0 \to B^{\vee} \to A^{\vee} \to \oplus \mathbb Z(p^{\infty}) \to 0 $. By the structure theorem of divisible groups $ A^{\vee}$ is isomorphic to the direct sum of copies of $ \mathbb Z(p^{\infty}) $ and $\mathbb Q$. But $A^{\vee}$ is torsion and hence $A \simeq \mathbb Z_p^{\infty}$.
\end{proof}
This shows that $\C^{0}_{K, p}$ depends only on $p$ and since isomorphic $T_q$ share the same $p$ we conclude that $\C^{0}_{K, p}$ is determined by the our data. 

\subsubsection{The non $p$-part}

Now we pick the prime number $l \ne p$ and consider the $l$-part $\C^{0}_{K, l}$ of $\C^{0}_{K}$. If $l$ is such that $\CL^{0}_{l} (K) \simeq \{0 \}$ then obviously $\mathcal T_{q,l} \simeq \C^{0}_{K,l}$. Let $l$ be a prime such that $\CL^{0}_{l} (K)$ is not trivial.  We know that the following sequence is totally non-split: 

$$1\to \mathcal T_{q,l} \to \C^{0}_{K, l}  \to \CL^{0}_{l} (K) \to 1 .$$

Fix a natural number $n$. Then multiplication by $l^n$ map induces the following commutative diagram:

\[\xymatrix{
1 \ar[r] & \mathcal T_{q,l}[l^{n}] \ar@{^(->}[d] \ar@{^(->>}[r]         & \C^{0}_{K, l}[l^{n}]\ar@{^(->}[d]\ar[r]^{0}  			&\CL^{0}_{l} (K)[l^n]\ar@{^(->}[d]  &\\
1 \ar[r] & \mathcal T_{q,l}\ar[d]^{l^n} \ar[r]          &\C^{0}_{K, l}\ar[d]^{l^n} \ar[r]    			        &\CL^{0}_{l} (K)\ar[d]^{l^n} \ar[r] &1 \\ 
1 \ar[r] & \mathcal T_{q,l}\ar[r]\ar@{->>}[d]                   & \C^{0}_{K, l} \ar[r] \ar@{->>}[d]			        & \CL^{0}_{l} (K)\ar[r] \ar@{->>}[d]  &1\\
 & \mathcal T_{q,l}/l^n\mathcal T_{q,l}\ar[r] & \C^{0}_{K, l}/l^n \C^{0}_{K, l}\ar[r]        &\CL^{0}_{l} (K)/l^n \CL^{0}_{l} (K)\ar[r] &1\\
}\]

Since our main sequence is totally non-split the map from $\C^{0}_{K, l}[l^{n}]$ to $\CL^{0}_{l} (K)[l^n]$ is the zero map and the map from $T_{q,l}[l^{n}]$ to $ \C^{0}_{K, l}[l^{n}]$ is isomorphism. Now applying the Pontryagin duality to the above diagram we get: 

\[\xymatrix{
1  & (\mathcal T_{q,l}[l^{n}] )^{\vee} \ar[l]          & (\C^{0}_{K, l}[l^{n}] )^{\vee}\ar@{^(->>}[l]   			&(\CL^{0}_{l} (K)[l^n])^{\vee} \ar[l]^{0} &  \\
1  & (\mathcal T_{q,l})^{\vee}\ar@{->>}[u] \ar[l]          &(\C^{0}_{K, l})^{\vee}\ar@{->>}[u] \ar[l]    			        &(\CL^{0}_{l} (K))^{\vee}\ar@{->>}[u] \ar[l] &1\ar[l] \\ 
1  & (\mathcal T_{q,l})^{\vee}\ar[l]\ar[u]^{l^n}                   & (\C^{0}_{K, l})^{\vee} \ar[u]^{l^n} \ar[l]			        & (\CL^{0}_{l} (K))^{\vee}\ar[l] \ar[u]^{l^n}  &1\ar[l]\\
  & (\mathcal T_{q,l}/l^n\mathcal T_{q,l})^{\vee} \ar@{^(->}[u] & (\C^{0}_{K, l}/l^n \C^{0}_{K, l})^{\vee}\ar@{^(->}[u] \ar[l]        &(\CL^{0}_{l} (K)/l^n \CL^{0}_{l} (K))^{\vee}\ar@{^(->}[u]\ar[l] &1\ar[l]\\
}\]

Because of the construction of $\mathcal T_{q,l}$ the group $(\mathcal T_{q,l})^{\vee}$ is isomorphic to the direct sum of cyclic groups $(\mathcal T_{q,l})^{\vee} \simeq \oplus_{k \ge N(l)} \oplus_{\mathbb N} \mathbb Z/l^{k}\mathbb Z$ and therefore $\cap_n l^{n} (\mathcal T_{q,l})^{\vee} = \{0 \}$. It means we have $(\cap_n l^{n} (\C^{0}_{K, l})^{\vee}) \subset (\CL^{0}_{l} (K))^{\vee}$. Our goal is to show that $(\cap_n l^{n} (\C^{0}_{K, l})^{\vee}) = (\CL^{0}_{l} (K))^{\vee}$.
\begin{lem}
Given any non-zero element $x$ of $ (\CL^{0}_{l} (K))^{\vee} \subset (\C^{0}_{K, l})^{\vee}$ and any natural number $n$ there exists element $c_x \in (\C^{0}_{K, l})^{\vee}$ such that $l^n c_x = x$. 
\end{lem}
\begin{proof}
For fixed $n$ consider the above diagram. Since the second row is exact the image of $x$ in $(\mathcal T_{q,l})^{\vee}$ is zero. Then its image in $(\mathcal T_{q,l}[l^{n}] )^{\vee}$ is also zero. Since $(\mathcal T_{q,l}[l^{n}] )^{\vee}  \simeq (\C^{0}_{K, l}[l^{n}] )^{\vee}$ it means that image of the non-zero element $x$ in $(\C^{0}_{K, l}[l^{n}] )^{\vee}$ is zero. Since the second column is exact this means that $x$ lies in the image of the multiplication by $l^n$ map from $(\C^{0}_{K, l})^{\vee} $ to $(\C^{0}_{K, l})^{\vee} $ and therefore there exists $c_x$ such that $l^{n} c_x = x$.
\end{proof}

It means that we have proved:
\begin{cor}\label{cor1}
The exact sequence $1\leftarrow (\mathcal T_{q,l})^{\vee} \leftarrow (\C^{0}_{K, l})^{\vee} \leftarrow (\CL^{0}_{l} (K))^{\vee} \leftarrow 1$ satisfies conditions of the theorem \ref{group}.
\end{cor}

In order to finish our proof we need to prove theorem \ref{group}.

\subsubsection{Proof of the Theorem \ref{group}}

First, let us recall the settings. 

\begin{theorem}
Let $\{ C_i \}$ be a countable set of finite cyclic abelian $l$-groups with orders of $C_i$ are not bounded as $i$ tends to infinity and let $A$ be any finite abelian $l$-group.
Then up to isomorphism there exists and unique torsion abelian $l$-group $B$ satisfying two following conditions:
\begin{enumerate}
	\item There exists an exact sequence: $1 \to A \to B \to \oplus_{i \ge 1} C_i \to 1$;
	\item $A$ is the union of all divisible elements of $B$: $A = \cap_{n \ge 1} nB$.
\end{enumerate} 
\end{theorem}
\begin{bf}Proof of the existence.
\end{bf}
Given a group $A$ and $\oplus_{i \ge 1} C_i $ let $k_i$ denotes the order of the group $C_i$. Because of the assumptions of the theorem, the sequence of orders $k_i$ is not bounded and hence for each natural number $N$ there exists $i$ such that $k_i \ge N$. Let us pick a sequence of indexes $j_i$, $i \in N$ such that $k_{j_i} \ge l^{i}$. Let $\alpha_0, \dots, \alpha_{n-1}$ be any finite set of generators of $A$. Consider the sequence $a_m$ of elements of $A$ defined as follows: $$a_{m} = \begin{cases} \alpha_{i \mod n}, & \mbox{if } m = j_i   \\ 0  , & \mbox{ otherwise.} \end{cases}$$

Consider the abelian group $B$ isomorphic to the quotient of the direct sum $A \oplus( \oplus_{i\in \mathbb N} X_i\mathbb Z)$ of countably many copies of $\mathbb Z$ and one copy of $A$ by the relations $k_i X_i = a_{i}$. We have that $B$ contains $A$ as a subgroup and the quotient of $B$ by $A$ is isomorphic to $\oplus_i C_i$. It means that the group $B$ satisfies the first condition of the theorem. Now, consider the group $Z = \cap_{n \ge 1} nB$. Obviously, $Z \subset A$ and we would like to show that actually $Z = A$. This follows from the fact that for any fixed number $N>1$ the set $\{ k_{j_i}X_{j_i} | i \ge \log_l{N} \}$ generates $A$ and satisfies $k_{j_i} \ge l^{i} \ge l^{\log_l(N)} \ge N$.

\begin{bf}Proof of the uniqueness.
\end{bf}
Suppose we are given an abelian torsion $l$-group $B$ which satisfies both conditions of the our theorem.  
Denote the map from $B$ to $ \oplus_{i \ge 1} C_i$ by $\phi$.
Let $\tilde{x}_i$ denotes a generator of the cyclic group $C_i$ and let $k_i$ denotes the order of $C_i$. Let $x_i$ be an element of $B$ such that $\phi(x_i) = \tilde{x}_i$, then $k_i x_i \in A$.

\begin{lem}
For any positive integer $M$ which is a power of $l$ the set $A_M = \{ k_i x_i | k_i \ge M \}$ generates $A$. 
\end{lem}
\begin{proof}
Without loss of generality we assume that $M\ge \# A$. Pick a non-zero element $a \in A$. Because of the second property $a$ could be written as $M^2 y$, where $y \in B$. Since the sequence  $1 \to A \to B \to \oplus_{i \ge 1} C_i \to 1$ is exact  we could write $y$ as finite combination of $x_{i_j}$ and elements of $A$: $y= b_{i_1}x_{i_1} + b_{i_2}x_{i_2} +\dots +b_{i_n}x_{i_n} + a_0$. Pick the subset $S$ of $i_1, \dots, i_n$ consisting of indexes of $i_j$ such that $k_{i_j} \ge M$. Since $M^2 x_{i_j} = 0$ if $k_{i_j} < M$ we have : $ M^2 \sum_{j \in S} b_{i_j}x_{i_j} = a$. On the other hand $0 =\phi(a) =  M^2 \sum_{j \in S} b_{i_j} \tilde{x}_{i_j}$ and hence $M^2 b_{i_j} $ is divisible by $ k_{i_j} $ and $a =  \sum_{j \in S} \frac{b_{i_j}M^2} {k_{i_j}} k_{i_j}x_{i_j} $. Which means that $\{ k_i x_i | k_i \ge M \}$ generates $A$.
\end{proof}
\begin{bf}Remark: \end{bf} consider the sequence $a_i = k_i x_i$ of elements of $A$ from the above lemma. We will say that this sequence $\langle a_i \rangle $ \emph{strongly generates }$A$.

 Note that $B$ as abstract abelian group is isomorphic to the group generated by elements $X_i$ and $a_i$ such that $k_i X_i = a_i$: $B = \langle X_i, a_i \rangle / (k_iX_i - a_i)$. Given another abelian group $B'$ satisfying conditions of our theorem we know that $B' = \langle X'_i, a'_i \rangle / (k_iX'_i - a'_i)$. If for any $i$ we have $a_i = a'_i$ as elements of $A$ then, obviously $B \simeq B'$. Our goal is to show that $B \simeq B'$ in any case.

\begin{defin}
Given two such groups $B$, $B'$ consider the set $S = \{ i | a_i = a'_i \}$. We will say that $B$ and $B'$ have \emph{large overlap} if the set $\{a_i | i \in S  \}$ strongly generates $A$, i.e. that for any integer $M$ the set $S_M = \{a_i | i \in S, k_i \ge M  \} $ generates $A$. 
\end{defin}

 We have the following observation:

\begin{lem}
If $B$ and $B'$ have large overlap, then they are isomorphic.
\end{lem}
\begin{proof}
For each index $i$ consider the difference $a_i - a'_i$. Since $B$ and $B'$ have large overlap, we could write this difference as finite sum $\sum_{m \in S} \lambda_m k_m X'_m$ with $k_m \ge k_i$. Since both $k_m$ and $k_i$ are powers of $l$ the ratio $\frac{k_m}{k_i}$ is an integer.
Consider the map $\psi$ from $B$ to $B'$ defined as follows. The map $\psi$ is identity on $A$. If $i \in S$ then $\psi( X_i ) = X'_i$, otherwise $\psi(X_i) = X'_i + \sum_{m \in S} \lambda_m \frac{k_m}{k_i} X'_m$. We claim that $\psi$ is a homomorphism: if $i \in S$ then $a_i = \psi(k_i X_i ) = k_i \psi (X_i) = k_i X'_i = a'_i $. If $i \not \in S$, we have $a_i = \psi(k_i X_i)= k_i (X'_i + \sum_{m \in S} \lambda_m \frac{k_m}{k_i} X'_m) =  k_i (X'_i) + \sum_{m \in S} \lambda_m k_m X'_m = a'_i + (a_i - a'_i) = a_i$. In other words it sends generators of $B$ to elements of $B'$ preserving all relations. 
We claim moreover that the map $\psi$ is an isomorphism since we will construct the inverse map $\phi$ from $B'$ to $B$ as follows. The map $\phi$ is identity on $A$. For $i \in S$ we have $\phi( X'_i) = X_i$ and for $i \not \in S$ we have $\phi(X'_i) = X_i - \sum_{m \in S} \lambda_m \frac{k_m}{k_i} X_m$. Then, for $i \not \in S$ we have: 
$$\phi( \psi(X_i) ) = \phi ( X'_i + \sum_{m \in S} \lambda_m \frac{k_m}{k_i} X'_m )  = \phi(X'_i) + \phi(\sum_{m \in S} \lambda_m \frac{k_m}{k_i} X'_m) = $$
$$ = (X_i - \sum_{m \in S} \lambda_m \frac{k_m}{k_i} X_m ) + (\sum_{m \in S} \lambda_m \frac{k_m}{k_i} X_m) = X_i.$$
  
\end{proof}

Now we will prove: 

\begin{cor}
Two groups $B$ and $B'$ satisfying conditions of the above theorem are isomorphic.
\end{cor}
\begin{proof}
Suppose that there exists a partition of the set of positive integers $\mathbb N$ on two sets $\mathbb N = I_1 \cup I_2$, $I_1 \cap I_2 =\emptyset$ such that each of the set $\{a_i | i \in I_1 \} $ and $\{ a'_i | i\in I_2 \}$ strongly generates $A$. Then we define abelian group $D$ to be the quotient of the direct sum $A \oplus( \oplus_{i\in \mathbb N} X_i\mathbb Z)$ of countably many copies of $\mathbb Z$ and one copy of $A$ by the relations $k_i X_i = a_{i}$, $i \in I_1$ and $k_i X_i = a'_i$, $i \in I_2$. Obviously $D$ also satisfies conditions of the above theorem. Moreover $D$ and $B$ and also $D$ and $B'$ have large overlap, therefore $B \simeq D \simeq B'$. 

Now we will show that such partition exists. 
We will construct this partition inductively. Let $N_0 =0$ and let $N_1$ be the minimal integer such that elements of the set $S_1 = \{ a_i  | i \le N_1$ and $k_i \ge l \}$ generate $A$. The reason for this number to exists is the following. The sequence $a_i$ strongly generates $A$ which implies that there exist indexes $i$ with $k_i \ge l$ such that $a_i$ generate $A$, but $A$ is a finite group and hence we could pick a finite number of elements with $k_i \ge l$ generating $A$. Note that dropping out finitely many indexes doesn't affect the fact that each of the sequences $a_i$ and $a'_i$ strongly generates $A$. Suppose we've constructed the number $N_m$ then let $N_{m+1}$ be a minimal integer such that elements of the set $$S_{m+1} = \begin{cases}  \{a'_i |  N_m < i \le N_{m+1}$ and $k_i \ge l^{m+1}  \}, & \mbox{if } m \mbox{ is odd}  \\ \{a_i |  N_m < i \le N_{m+1} $ and $k_i \ge l^{m+1}  \}, & \mbox{ otherwise.} \end{cases}$$  generate $A$. Finally, we define $I_1 = \cup_{m \ge 0} \{i \in \mathbb N | N_{2m} < i \le N_{2m+1}   \}$ and $I_2 = \cup_{m \ge 1} \{i \in \mathbb N | N_{2m-1} < i \le N_{2m}   \}$.
\end{proof}

\section{Proof of Corollaries}
In this section we will prove corollaries \ref{corl1}, \ref{corl2} and \ref{corl3}. The first two will follow from the existence for a given constant field $k=\mathbb F_q$ an elliptic curve $E$ over $k$ with the group $E (\mathbb F_q)$ of $\mathbb F_q$-rational points having order $q$, since in the case of elliptic curves we have $E (\mathbb F_q) \simeq \CL^{0}(K_{E})$, where $K_{E}$ denotes the associated to $E$ global function field.

\begin{defin}
Fix a finite field $\mathbb F_q$. Let $N$ be an integer number in the Hasse interval: $N \in [ -2\sqrt{q} ; 2\sqrt{q} ]$. We will call it admissible if there exists an elliptic curve $E$ over $\mathbb F_q$ with $q+1 - \#  E( \mathbb F_q) = N$.  
\end{defin}

The following statement is a part of the classical statement due to Waterhouse, for reference see ~\cite{Schoof}: 
\begin{theorem}[Waterhouse] 
If $\GCD(p,N)=1$ then the number $N$ is admissible. 
\end{theorem}
\begin{cor}
Given a finite field $\mathbb F_q$ there exists an elliptic curve $E$ over $\mathbb F_q$ with $\# E(\mathbb F_q) = q$.
\end{cor}

The above remarks finish the proof of corollaries \ref{corl1} and \ref{corl2}.  Now we will discuss the proof of the corollary \ref{corl3}. Our goal is to show :
\begin{theorem}\label{maincor}
Given a constant field $k=\mathbb F_q$ with characteristic $p \ne 2$ there are infinitely many non-isomorphic curves $X$ over $k$ with different two-part of the group of $k$-rational points on the Jacobian varieties associated to them. 
\end{theorem}
\begin{proof}
For any positive integer $n$ there exists a monic irreducible polynomial of degree $n$ with coefficients in $\mathbb F_q$. Let us pick any sequence of such polynomials $D_n(x)$, $n \in \mathbb N$ with the property that $\deg(D_{n+1}(x) ) > \deg(D_n(x))$ and $\deg(D_1(x)) \ge 3$. Consider the family of affine curves defined by the equation $C_m: y^2 = D_1(x)D_2(x) \dots D_m(x)$. Since $D_i$, $i \in \mathbb N$ are mutually distinct these affine curves are smooth. Let $X_m$ denotes the normalization of the projective closure of $C_m$. Then $X_m$ is a hyper-elliptic curve of the genus $g_m = \lfloor\frac{\deg(D_1(x)) + \dots + \deg(D_m(x)) -1 } {2} \rfloor$. The Weil-bound insures that the order of the group of $\mathbb F_q$-rational points of the Jacobian variety $J_m$ associated to $X_m$ satisfies the following: $$ (\sqrt{q} - 1)^{2g_m}  \le \#J_m(\mathbb F_q) \le (\sqrt{q} +1 )^{2g_m}, $$
and therefore the two-part of $J_m(\mathbb F_q)$ is bounded from above by $(\sqrt{q} +1 )^{2g_m}$. 
On the other hand theorem 1.4 from \cite{cornelhyp} states that the  two-rank of $J_m(\mathbb F_q)$ is at least $m-2$. Therefore, among the family $X_m$ there are infinitely many curves with different two-part of the group $J_m(\mathbb F_q) $ and therefore their function fields $K_m$ have non-isomorphic $\mathcal G^{ab}_{K_m}$.
\end{proof}
\newpage

\bibliography{mybib}{}
\bibliographystyle{plain}

\newpage

\tableofcontents

\end{document}